\documentclass{article}
\usepackage[utf8]{inputenc}
\usepackage{amsmath,amsfonts, amssymb, amsthm, amscd}
\usepackage{tikz,tikz-cd}
\usepackage{xcolor,url}
\usepackage{titlesec, eucal}
\usepackage{lineno}
\titleformat{\subsubsection}[runin]{\normalfont\bfseries}{\thesubsubsection.}{3pt}{}
\usepackage{multicol}

\newcommand*{\DashedArrow}[1][]{\mathbin{\tikz [baseline=-0.25ex,-latex, dashed,#1] \draw [#1] (0pt,0.5ex) -- (1.3em,0.5ex);}}%

\newtheorem{proposition}[subsubsection]{Proposition}
\newtheorem{theorem}[subsubsection]{Theorem}
\newtheorem{lemma}[subsubsection]{Lemma}
\newtheorem{corollary}[subsubsection]{Corollary}
\theoremstyle{definition}
\newtheorem{definition}[subsubsection]{Definition}

\newtheorem{problem}{Problem}
\theoremstyle{plain}
\newtheorem{thm}{Theorem}

\newcommand{\C}{\mathbb C}

\newcommand{\Z}{\mathbb Z}
\newcommand{\Q}{\mathbb Q}

\newcommand{\R}{\mathbb R}
\renewcommand{\O}{\mathcal O}

\newcommand{\di}{\partial}

\renewcommand{\P}{\mathbb{P}}
\newcommand{\F}{\mathbb F}

\newcommand{\Tot}{\operatorname{Tot}}

\newcommand{\Bl}{\operatorname{Bl}}

\newcommand{\diag}{\operatorname{diag}}
\renewcommand{\phi}{\varphi}
\newcommand{\V}{\mathrm{Var}}
\title{Complex surfaces with many algebraic structures}
\author{Anna Abasheva, Rodion D\'eev}
\date{}

\begin{document}

\maketitle

\begin{abstract}

We find new examples of complex surfaces with countably many non-isomorphic algebraic structures. Here is one such example: take an elliptic curve $E$ in $\mathbb P^2$ and blow up nine general points on $E$. Then the complement $M$ of the strict transform of $E$ in the blow-up has countably many algebraic structures. Moreover, each algebraic structure comes from an embedding of $M$ into a blow-up of $\mathbb P^2$ in nine points lying on an elliptic curve $F\not\simeq E$. We classify algebraic structures on $M$ using a {\bf Hopf transform}: a way of constructing a new surface by cutting out an elliptic curve and pasting a different one. Next, we introduce the notion of an {\bf analytic K-theory of varieties}. Manipulations with the example above lead us to prove that classes of all elliptic curves in this K-theory coincide. To put in another way, all motivic measures on complex algebraic varieties that take equal values on biholomorphic varieties do not distinguish elliptic curves.

\end{abstract}

\tableofcontents

\newpage

\section{Introduction}

A compact complex analytic variety $X^{an}$ has at most one algebraic structure; that is to say, there is at most one algebraic variety $Y$ such that $Y^{an} \cong X^{an}$. However, this is false in the non-compact world. We do not have general techniques to describe the set of algebraic structures on a non-compact complex variety. Available results show that a variety can have infinitely many of algebraic structures. We briefly review examples in dimension two.

\begin{itemize}
    \item Let $C$ be an affine curve of positive genus. Then $C^{an} \times \C$ admits uncountably many non-isomorphic affine structures \cite{Jelonek}. Indeed, every line bundle on $C$ is analytically trivial, hence its total space is biholomorphic to $C^{an}\times \C$. 
    \item The surface $\C^* \times \C^*$ is biholomorphic to the unique non-trivial affine bundle of rank one and degree zero over any elliptic curve \cite[Thm.\:6.12]{Neeman}. Hence, it has uncountably many non-isomorphic algebraic structures. 
    \item Let $C$ be a smooth projective curve and $p\colon X\to C$ an algebraic elliptic surface. Choose a smooth fiber $E = p^{-1}(x)$. For every $n>0$, we can find an elliptic fibration $p_n\colon X_n\to C$, called a {\bf log-transform} of $p$, that has a fiber of multiplicity $n$ over $x$ and is biholomorphic to $p$ over $C-\{x\}$. The new surfaces $X_n$ are algebraic unless $X\simeq E\times C$ \cite[Ch.\:I, Thm.\:6.12]{Friedman_Morgan}. Hence the surface $X-E$ has at least countably many algebraic structures.
    \item We can construct even more algebraic structures on the surface $X-E$ from the previous example using Shafarevich--Tate twists. We refer the reader to \cite[Ch.\:I, Sections 1.5.1, 1.5.3]{Friedman_Morgan} for more details.
\end{itemize}

We complement this list by the following beautiful example.

\begin{thm}[Theorem \ref{many-algebraic-structures}, Corollary \ref{Hopf-transform-preserves-rationality}]\begin{enumerate}
\label{intro_theorem_A}
    \item Pick nine points in $\P^2$ in sufficiently general position. Let $X$ be the blow-up of $\P^2$ in these points and $E$ the strict transform of the unique elliptic curve passing through them. Then the complement of $E$ in $X$ admits countably many algebraic structures.
    \item More generally, let $X$ be a compact rational surface containing a square-zero elliptic curve $E$. Assume that the normal bundle to $E$ in $X$ is sufficiently general. Then the surface $X-E$ admits countably many algebraic structures.

\end{enumerate}

\end{thm}

In addition, in the Case 1 every algebraic structure on $X-E$ arises from a holomorphic embedding into a blow-up $Y$ of $\P^2$ in some other nine points. The complement of $X-E$ in $Y$ is an elliptic curve $F$ which is remarkably {\em not isomorphic and not even isogenous} to $E$. In the algebraic realm, this is only possible for singular curves (\cite{Blanc_Poloni_van_Santen} and references therein). Indeed, if $S$ and $S'$ are smooth algebraic surfaces containing smooth curves $C$ and $C'$ respectively, then an algebraic isomorphism between $S-C$ and $S'-C'$ implies an isomorphism between $C$ and $C'$.

\hfill

Our construction in Theorem \ref{intro_theorem_A} is fundamentaly different from those in the list above. We developed a method to manufacture new surfaces from a surface with a square-zero elliptic curve that we call a {\bf Hopf transform} (Definition \ref{Hopf_transform}). It serves as a counterpart to a log-transform in the case of non-elliptic surfaces.

\hfill

We can summarize a Hopf transform as follows: take an elliptic curve on a surface, cut it out, and paste another elliptic curve instead. Thanks to a theorem of Arnold and Ueda (Theorem \ref{Arnold}), a square-zero elliptic curve $E \subset S$ whose normal bundle $L$ is sufficiently general has a holomorphic tubular neighbourhood. That is to say, $S$ is locally biholomorphic to $\Tot(L)$ along $E$. We will clarify the meaning of ``sufficiently general'' in Definition \ref{diophantine-definition}; for now, it is enough to say that all ``sufficiently general'' line bundles are non-torsion. Using Hopf surfaces, i.e., non-K\"ahler surfaces of the form $\C^2-\{0\}/\diag(\lambda,\mu)$ and their finite quotients, we show that there exist countably many degree zero line bundles $L'\to F$ over different elliptic curves $F$ such that $\Tot(L)^{an} - E \cong \Tot(L')^{an} - F$. We call such line bundles $L$ and $L'$ {\bf analytically cobordant} (Definition \ref{def of duals and cobordant}). A Hopf transform, which can be interpreted as a gluing of a surface with an appropriate Hopf surface, replaces an elliptic curve $E\subset S$ with $F$. Moreover, we are free to choose from a countable set of candidates for $F$. We apply this surgery to a blow-up of nine points on $\P^2$ and show that the result is again a blow-up of $\P^2$ along some other nine points. It yields Theorem \ref{many-algebraic-structures}.

\hfill

Our notion of a Hopf transform was inspired by Koike and Uehara's construction in \cite{Koike_Uehara_non_proj}. They use the theorem of Arnold and Ueda mentioned above to glue two open subsets of two different blow-ups of $\P^2$ in nine points into a K3 surface.

\hfill

We apply Theorem \ref{intro_theorem_A} to study the {\bf analytic Grothendieck group} $K_0^{an}$ of varieties (Definition \ref{analytic-grothendieck-definition}). It is the quotient of the Grothendieck group $K_0(\V_\C)$ of complex algebraic varieties\footnote{$K_0(\V_\C)$ is the abelian group generated by classes of algebraic varieties over $\C$ modulo scissor relations: $[X]-[Y] = [X-Y]$, where $Y$ is a closed subvariety of $X$.} by additional relations of the form
$$
[X] = [Y] \:\:\:\text{     if     }\:\:\: X^{an}\simeq Y^{an}.
$$
While $K_0(\V_\C)$ has attracted a considerable amount of attention (see \cite{Bittner,Kontsevich_Tschinkel_specialization,Nicaise_Shinder} among many others), the notion of $K_0^{an}$ is, to our knowledge, new. The group $K_0^{an}$ is strikingly different from its algebraic counterpart as the theorem below illustrates.

\begin{thm}[Theorem \ref{analytic K zero}]\label{intro_theorem_B}
    All elliptic curves have the same classes in $K_0^{an}$. 
\end{thm}

\noindent In contrast, non-isomorphic curves have different classes in $K_0(\V_\C)$ \ref{Bittner-theorem}.

\hfill

This paper is organized as follows. In Section \ref{hopf-transforms-section}, we prove that the total space of a non-torsion degree-zero line bundle on an elliptic curve has countably many compactifications by different elliptic curves (Theorem \ref{embeddings into all hopfs}). Then we introduce the notions of Hopf duality and analytic cobordance (Defintion \ref{def of duals and cobordant}), and give the definition of Hopf transforms (Definition \ref{Hopf_transform}). In Section \ref{square-zero-section}, we partially classify surfaces to which Hopf transforms apply, that is, surfaces with square-zero elliptic curves with non-torsion normal bundle. With this classification in mind, we prove Theorem \ref{intro_theorem_A}, the main result of our paper. Finally, in Section \ref{grothendieck-section}, we introduce the analytic Grothendieck group and prove Theorem \ref{intro_theorem_B}. Towards the end, we pose several intriguing questions about the analytic Grothendieck group.


\section{Hopf transforms}\label{hopf-transforms-section}

\subsection{Analytic compactifications}


\begin{definition}
    Let $M$ be a non-compact complex surface. A pair $(S, C)$ of a smooth compact complex surface $S$ and (possibly singular) curve $C \subset S$ is called an {\it analytic compactification} or just a {\it compactification of $M$} if $S-C$ is biholomorphic to $M$. A compactification is called {\em minimal} if none of the curves in the support of $C$ are exceptional.
\end{definition}

The following classification theorem is due to Enoki \cite[Theorem C-I$_0$]{Enoki82}.


\begin{theorem}\label{enoki-classification}
    Let $\Tot(L)$ be the total space of a line bundle $L$ of degree zero over an elliptic curve $E$. Then every minimal compactification $(S,C)$ of $M$ is of the following form:
    \begin{enumerate}
        \item $S\cong \P({\cal O}\oplus L)$ is a $\P^1$-bundle over an elliptic curve; $C$ is its section, and $(C)^2=0$.
        \item $S$ is a Hopf surface (Definition \ref{Hopf_surface_def}), and $C$ is an elliptic curve.
    \end{enumerate}
\end{theorem}

If $L$ is non-torsion, the set of Hopf surfaces compactifying $\Tot(L)$ is countable, as we will see in Theorem \ref{embeddings into all hopfs}. For now, let us note that all such compactifications contain at least two elliptic curves.


\subsection{Hopf surfaces}

\begin{definition}
\label{Hopf_surface_def}\cite[Sect.\:10]{KodairaII}
    A {\it Hopf surface} is a compact complex surface whose universal cover is $\C^2-\{0\}$. A Hopf surface is called {\em primary} if its fundamental group is $\Z$, otherwise it is called {\em secondary}.
\end{definition}


\subsubsection{}\label{subsubsection primary Hopfs non linear}
All Hopf surfaces are non-K\"ahler. By \cite[Sect.\:10, p.\:695]{KodairaII} every primary Hopf surface $X$ is biholomorphic to a quotient $\C^2-\{0\}/ \Gamma$ where $\Gamma$ is either a linear diagonal operator or of the form
$$
\Gamma(x,y) = (\alpha^nx + y^n, \alpha y).
$$
In the second case, a Hopf surface has only one irreducible curve: the image of the $x$-axis. Therefore, such surfaces do not arise as compactifications of $\Tot(L)$. 


\subsubsection{}
\label{subsubsection primary Hopfs linear}
Let us study the case when $\Gamma = \diag(\lambda,\mu)$. In order for $\C^2-\{0\}/\Gamma$ to be compact, the absolute values of the eigenvalues must be either both greater than one or both less than one. By replacing $\Gamma$ with $\Gamma^{-1}$ if necessary, we may assume the latter. 

\begin{definition}
    A {\em diagonal Hopf surface} is $\C^2-\{0\}/ \diag(\lambda, \mu)$ for $|\lambda|<1$, $|\mu|<1$. We will denote it by $H(\lambda,\mu)$
\end{definition}


\subsubsection{}
One can show that two Hopf surfaces $H(\lambda, \mu)$ and $H(\lambda',\mu')$ are isomorphic if and only if $\lambda = \lambda'$ and $\mu = \mu'$ (up to a permutation). All automorphisms of $H(\lambda,\mu)$ are induced by a linear diagonal automorphism of $\C^2$.


\subsubsection{}\label{generalities on primary hopfs}
If $\lambda^n = \mu^m$ for some non-zero integers $n$ and $m$ then $H(\lambda,\mu)$ is an elliptic surface with smooth but possibly non-reduced fibers \cite[Ch.\:V, Prop.\:18.2]{Barth_Hulek_Peters_Van_de_Ven}. All elliptic curves in $X$ are fibers, hence their normal bundles are torsion. Such a surface cannot be a compactification of $\Tot(L)$ for a non-torsion line bundle $L$.

When $\lambda^n\ne\mu^m$ for any pair of non-zero integers the surface $H(\lambda,\mu)$ contains exactly two irreducible curves \cite[Ch.\:V, Prop.\:18.2]{Barth_Hulek_Peters_Van_de_Ven}. They are the images of the axes under the quotient map and are isomorphic to $E_\lambda:=\C^*/\lambda$ and $E_\mu:=\C^*/\mu$. They have both square zero. 

The complement to $E_\mu$ in $H(\lambda,\mu)$ is $\C^*\times \C/\diag(\lambda,\mu)$. It is the total space of a line bundle over $E_\lambda$, and the map to $E_\lambda$ is the projection to the first coordinate.


\subsubsection{} Every elliptic curve can be written as $\C/\Z + \Z\tau$ for $\operatorname{Im}(\tau)>0$. The parameter $\tau$ is defined modulo the $\operatorname{SL}_2(\Z)$-action. The map $x\mapsto e^{2\pi ix}$ induces an isomorphism
$$
\C/\Z + \Z\tau \stackrel{\sim}{\to} \C^*/q^\tau.
$$
Here we use the shorthand $q^\tau$ for $e^{2\pi i\tau}$. The total space of every line bundle of degree zero over $\C/\Z + \Z\tau$ is of the form
\begin{equation}\label{total_of_L}
    \mathrm{L}(\tau,A,B) = \mathbb C^2/\sim:\:\:\:\:\:\: (x,y) \sim (x+1, Ay);\: (x,y)\sim (x+\tau, By),
\end{equation}
where $A,B\in U(1)$. These numbers are uniquely defined once we fix $\tau$. Similarly, the total space of a line bundle of degree zero over $E_\lambda = \C^*/\lambda$ can be written as
\begin{equation}
    \mathcal{L}(\lambda,\mu) = \C^*\times \C/\diag(\lambda,\mu),
\end{equation}
where $\lambda,\mu\in \C^*$. 


\subsubsection{}\label{when L(lambda,mu) are isomorphic}The class of $\mathcal{L}(\lambda,\mu)$ in $\operatorname{Pic}^0(E_\lambda)\simeq E_\lambda$ is $\mu\pmod{\lambda^\Z}$ \cite[Sect.\:10, p.\:696]{KodairaII}. In particular, line bundles $\mathcal{L}(\lambda,\mu)$ and $\mathcal{L}(\lambda,\mu')$ are isomorphic if and only if $\mu' = \mu\lambda^n$.


\subsubsection{}
\label{main_biholomorphism}
A line bundle on an elliptic curve can be represented as $\mathrm{L}(\tau,A,B)$ or $\mathcal{L}(\lambda,\mu)$. Let us figure out how to switch between the two descriptions. Write $A = q^u$ for a number $u\in \R$. Then we have a linear isomorphism of line bundles

$$
\alpha_{\tau, u, B}\colon \mathrm{L}(\tau, q^u,B) \stackrel{\sim}\to \mathcal L(q^\tau, q^{-u\tau}B)
$$
given by 
\begin{equation}
\label{from tau to hopf}
        \alpha_{\tau, u, B}(x,y) = (q^{x}, q^{-ux}y).
\end{equation}

Note that $|q^\tau|<1$ and $|q^{-u\tau}B|<1$ if $u$ is chosen to be negative. Given a line bundle $\mathcal L(q^\tau, \mu)$, we can find corresponding $u$ and $B$. Indeed, write $\mu = q^\sigma$ and $B = q^v$. The equation
$$
\sigma = u\tau + v
$$
determines $u$ and $v$ uniquely since $u$ and $v$ are real numbers. Different choice of $\sigma$ does not change $B = q^v$.


\subsubsection{}\label{subsection on sl2 action} The isomorphism (\ref{from tau to hopf}) depends only on the choice of $\tau$ and $u$ such that $q^u=A$. Let us see what happens if we change $\tau$. There is a linear isomorphism of line bundles
$$
\beta_\Gamma \colon \mathrm{L}(\tau, q^u,q^v)\stackrel{\sim}\to\mathrm{L}\left(\frac{k\tau + l}{m\tau +n}, q^{mv+nu}, q^{kv+lu}\right)
$$
for $\begin{pmatrix}k & l\\m & n\end{pmatrix}\in\operatorname{SL}_2(\Z)$. The map (\ref{from tau to hopf}) identifies $\mathrm{L}(\tau, q^u,q^v)$ with $\mathcal L(\lambda', \mu')$, where
\begin{gather*}
   \lambda' = q^{\frac{k\tau+l}{m\tau+n}},\\ 
   \mu' = q^{-(mv+nu+r)\frac{k\tau+l}{m\tau+n} + kv+lu},
\end{gather*}
and $r\in \Z$. Using that $nk-ml = 1$, we can rewrite the exponent in $\mu'$ as
$$
-(mv+nu+r)\frac{k\tau+l}{m\tau+n} + kv+lu = \frac{v-u\tau + r(k\tau + l)}{m\tau + n}.
$$
The elliptic curve in the complement to $\Tot(L)$ in $H(\lambda',\mu')$ is 
$$
E_{\mu'} \simeq \C/\Z + \Z\cdot \left(\frac{v-u\tau + r(k\tau + l)}{m\tau + n}\right)
$$
If $L$ is generic, different choices of $\begin{pmatrix}k & l\\m & n\end{pmatrix}$ and $r$ lead to different elliptic curves $E_{\mu'}$.


\subsubsection{}\label{translations} Consider a translation $t_s\colon x\mapsto x+s$ on the curve $\C/\Z+\Z\tau$. The line bundle $\mathrm{L}(\tau, A,B)$ has degree zero, hence it is isomorphic to its pullback along the translation. The isomorphism can be written explicitly as
\begin{gather*}
    t_s^*\colon \mathrm{L}(\tau, A, B)\to \mathrm L(\tau, A, B)\\ (x,y)\mapsto (x+s,y).
\end{gather*}
When we write $\mathrm{L}(\tau, A,B)$ as $\mathcal L(q^{\tau}, q^{-u\tau}B)$ as in \ref{main_biholomorphism}, the translation turns into the map
\begin{gather*}
 t_s^*\colon \mathcal L(q^{\tau}, q^{-u\tau}B) \to \mathcal L(q^{\tau}, q^{-u\tau}B)\\
 (x,y)\mapsto (q^s, q^{-us}).   
\end{gather*}
This map extends to the automorphism $\diag(q^s,q^{-us})$ of the Hopf surface $H(q^\tau, q^{-u\tau}B)$.


\subsubsection{}\label{automorphisms fixing zero} Consider the automorphism $f$ of $E$ sending $x$ to $-x$. The line bundle $f^*\mathrm{L}(\tau, A, B)$ is isomorphic to $\mathrm{L}(\tau, A^{-1}, B^{-1})$. In general, an automorphism $f$ of $E$ fixing $0$ induces a linear biholomorphism 
$$
f^*\colon \mathrm{L}(\tau, A, B)\stackrel{\sim}\to \mathrm{L}(\tau, A'', B''),
$$
where the pair of numbers $(A'', B'')$ does not necessarily coincide with $(A,B)$.


\subsubsection{}\label{multiplication by a constant} Every automorphism of $\mathcal L(\lambda,\mu)$ as a line bundle is the fiberwise multiplication by a constant $c$. It extends to the automorphism $\diag(1,c)$ of the Hopf surface $H(\lambda, \mu)$.


\begin{proposition}
\label{which hopfs arise}
Let $L$ be a degree-zero non-torsion line bundle on an elliptic curve $E$. Then the set of equivalence classes of open embeddings $\Tot(L)\stackrel{\iota}\to H$ into a primary Hopf surface is countable. Here, we call two embeddings $\iota\colon \Tot(L)\to H$ and $\iota'\colon \Tot(L)\to H'$ equivalent if there is an automorphism $g$ of $\Tot(L)$ and a biholomorphism $h\colon H\to H'$ such that the following diagram commutes:

\begin{center}
\begin{tikzcd}
\operatorname{Tot}(L) \arrow[r, "\iota", hook] \arrow[d, "g"] & H \arrow[d, "h"] \\
\operatorname{Tot}(L) \arrow[r, "\iota'", hook]               & H'              
\end{tikzcd}
\end{center}

\end{proposition}
\begin{proof}
Suppose $E=X/\Z+\Z\tau$ and $L$ is such that $\Tot(L)=\mathrm{L}(\tau,A,B)$. Consider an open embedding $\iota\colon \mathrm{L}(\tau, A, V)\to H$. By \ref{subsubsection primary Hopfs non linear} and \ref{generalities on primary hopfs}, the surface $H$ must be of the form $H(\lambda,\mu)$, and $\mathrm{L}(\tau,A,B)$ is mapped onto $H(\lambda,\mu)-E_\mu\simeq \mathcal L(\lambda,\mu)$. We saw in \ref{main_biholomorphism} that there exists a biholomorphism $$\alpha_{\tau',u',B'}\colon \mathrm{L}(\tau',A',B')\stackrel{\sim}\to \mathcal L(\lambda,\mu)$$
for some $\tau'\in\C$ such that $\mathrm{Im}~\tau'>0$ and $A',B'\in U(1)$.

The elliptic curves $E = \C/\Z+\Z\tau$ and $\C/\Z+\Z\tau'$ must be isomorphic, hence $\tau = \Gamma\cdot\tau'$ for a matrix $\Gamma\in\operatorname{SL}_2(\Z)$. In particular, the set of possible $\tau'$ is countable. The line bundle $\mathrm{L}(\tau', A', B')$ is linearly isomorphic to $\mathrm{L}(\tau, A'',B'')$ for some $A'', B''\in U(1)$ (\ref{subsection on sl2 action}). Therefore, the embedding $\mathrm{L}(\tau, A,B)\stackrel{\iota}\to H(\lambda,\mu)$ is the composition of the following maps:

\begin{center}\begin{tikzcd}
{\mathrm{L}(\tau, A,B)} \arrow[r, "\varphi"] & {\mathrm{L}(\tau, A'',B'')} \arrow[r, "\beta_\Gamma"] & {\mathrm L(\tau, A', B')} \arrow[r, "{\alpha_{\tau',u',B'}}"] & {\mathcal L(\lambda,\mu)} \arrow[r, hook] & {H(\lambda,\mu)},
\end{tikzcd}\end{center}

where $\mathrm{L}(\tau, A,B)\stackrel{\phi}\to \mathrm{L}(\tau, A'',B'')$ is some biholomorphism. 

\hfill

It remains to classify biholomorphisms $\mathrm{L}(\tau, A,B)\stackrel{\phi}\to \mathrm{L}(\tau, A'',B'')$ modulo equivalence. Let $f$ be the restriction of $\phi$ to the zero section $E\subset \mathrm{L}(\tau, A,B)$. It is a composition of a translation and an automorphism preserving zero. Two biholomorphisms $\mathrm{L}(\tau, A,B)\to\mathrm{L}(\tau, A'',B'')$ that differ by a translation are equivalent because a translation extends to an automorphism of $H(\lambda,\mu)$ (\ref{translations}). Hence, we may assume that $f\colon E\to E$ preserves zero. We have an isomorphism of line bundles $\mathrm{L}(\tau,A,B)\simeq f^*\mathrm{L}(\tau,A'',B'')$ as in \ref{automorphisms fixing zero}. The set of automorphisms fixing zero is finite, hence the number of pairs $(A'',B'')$ such that $\mathrm{L}(\tau, A,B)$ is isomorphic to $\mathrm{L}(\tau, A'',B'')$ is also finite. We have shown that every embedding $\mathrm{L}(\tau, A,B)\stackrel{\iota}\to H(\lambda,\mu)$ is equivalent to the composition 

\begin{center}
\begin{tikzcd}{L(\tau, A,B)} \arrow[r, "\psi"] & {\mathrm{L}(\tau, A,B)} \arrow[r, "f^*"] & {\mathrm{L}(\tau, A'',B'')} \arrow[r, "\beta_\Gamma"] & {\mathrm L(\tau, A', B')} \arrow[r, "{\alpha_{\tau',u',B'}}", hook] & {H(\lambda,\mu)},\end{tikzcd}\end{center}

where $\psi$ is a biholomorphism of $\mathrm{L}(\tau, A,B)$ identical on $E$. We will see later in Lemma \ref{open embeddings} that every such biholomorphism is multiplication by a constant. It extends to an automorphism of $H(\lambda,\mu)$ by \ref{multiplication by a constant}. Therefore, every embedding is equivalent to the composition of morphisms
\begin{center}
\begin{tikzcd}
{\mathrm{L}(\tau, A,B)} \arrow[r, "f^*"] & {\mathrm{L}(\tau, A'',B'')} \arrow[r, "\beta_\Gamma"] & {\mathrm L(\tau, A', B')} \arrow[r, "{\alpha_{\tau',u',B'}}", hook] & {H(\lambda,\mu)}
\end{tikzcd}
\end{center}
Such embedding depends only on the choice of an automorphism $f\colon E\to E$ fixing zero (finite number of choices), $\tau'$ in the $\operatorname{SL}_2(\Z)$-orbit of $\tau$ (countable number of choices) and $u'<0$ such that $A'=q^{u'}$ (countable number of choices). Our claim follows.
\end{proof}



\subsection{Secondary Hopf surfaces}


\subsubsection{}\label{secondary-with-non-torsion}
We have just classified primary Hopf surfaces that arise as compactifications of $\Tot(L)$. We will focus in this subsection on compactifications that are secondary Hopf surfaces. Every secondary Hopf surface is a quotient of a primary Hopf surface by a free action of a finite group \cite[Sect.10, p.695]{KodairaII}. Assume that $S$ is a secondary Hopf surface containing an elliptic curve with non-torsion normal bundle. Then $S$ is a quotient of a diagonal Hopf surface $H(\lambda,\mu)$ by an operator $\diag(q^{1/n},q^{r/n})$, where $q^{1/n}$ and $q^{r/n}$ are primitive $n$-roots of unity \cite[Sect.\:10, Thm.\:32]{KodairaII}.

\begin{theorem}
\label{embeddings into all hopfs}
    Let $L$ be a degree-zero non-torsion line bundle on an elliptic curve $E$. Then the set of equivalence classes of open embeddings $\Tot(L)\stackrel{\iota}\to H$ of $\Tot(L)$ into a Hopf surface (primary or secondary) is countable.
\end{theorem}
\begin{proof}

We know the claim for embeddings into primary Hopf surfaces (Proposition \ref{which hopfs arise}), hence it is enough to assume that $H$ is a secondary Hopf surface. By \ref{secondary-with-non-torsion}, every secondary Hopf surface compactifying $\Tot(L)$ must be of the form $H/\diag(q^{1/n},q^{r/n})$ for $(r,n)=1$, where $H$ is a diagonal Hopf surface. Given an embedding of $\Tot(L)$ into $H/\diag(q^{1/n},q^{r/n})$, we can construct an embedding of $\Tot(f^*L)$ to $H$, where $f\colon E'\to E$ is a quotient by an $n$-torsion element of $E'$. The set of possible $f$ is countable, and so is the set of equivalence classes of embeddings of $\Tot(f^*L)$ into a primary Hopf surface, hence the claim.
\end{proof}

\subsection{Hopf duality and analytic cobordance}

\subsubsection{} Let $L \to E$ be a non-trivial degree-zero line bundle. The ruled surface $\P(\O_E \oplus L)$ contains two sections: $\P(\O)$ and $\P(L)$. The map $\Tot(L) \to \P(\O\oplus L)$, $v \mapsto [1 : v]$ defines a biholomorphism between $\Tot(L)$ and $\P(\O\oplus L) - \P(L)$. Similarly, $\P(\O\oplus L) - \P(\O)$ is biholomorphic to $\Tot(L^{-1})$. Hence the surface $\P(\O_E\oplus L)$ is a compactification of total spaces of dual line bundles: $\Tot(L)$ and $\Tot(L^{-1})$. The analogous property of the diagonal primary Hopf surfaces motivates the following definition.

\begin{definition}
\label{def of duals and cobordant}
    Let $L \to E$ and $L' \to E'$ be two degree-zero line bundles on elliptic curves. They are called 
    \begin{itemize}
    \item {\bf Hopf dual} if there exists a Hopf surface $H$ with two elliptic curves $E_\lambda$ and $E_\mu$ such that $H - E_\mu\simeq \Tot(L)$ and $H - E_\lambda\simeq \Tot(L')$;
    \item {\bf analytically cobordant} if $L^{-1}\to E$ and $L'\to E'$ are either Hopf dual, or $E=E'$ and $L=L'$.
    \end{itemize}
\end{definition}

By Proposition \ref{which hopfs arise}, any non-torsion line bundle admits a countable number of Hopf dual and analytically cobordant line bundles. The base curves of analytically cobordant line bundles are in general non-isomorphic (\ref{subsection on sl2 action}). 

\subsubsection{}\label{Hopf-duals}
    Let $L \to E$ and $L' \to E'$ be Hopf dual line bundles. Then the spaces $\Tot(L) - 0_L$ and $\Tot(L') - 0_{L'}$ are biholomorphic, where $0_{\Xi} \subset \Tot(\Xi)$ denotes the zero section. Indeed, both spaces are biholomorphic to $H(\lambda,\mu) - \left(E_\lambda\cup E_\mu\right)$. The biholomorphism sends the neighborhood of the zero section in one to the neighborhood of the infinity section in the other and vice versa. Similarly, let $L \to E$ and $L' \to E'$ be analytically cobordant line bundles. Then $\Tot(L) - 0_L$ and $\Tot(L') - 0_{L'}$ are biholomorphic so that a neighbourhood of the zero section in one maps to the neighbourhood of the zero section in the other.

\subsubsection{} Of course, if $E \not\cong E'$, the biholomorphism $\Tot(L) - 0_L \simeq \Tot(L') - 0_{L'}$ cannot be extended to the zero section. To internalize it, let us understand the behavior of fibers of the bundle under this biholomorphism. Let $H = H(\lambda,\mu)$ be a Hopf surface. The fibers of $H - E_\mu = \mathcal L(\lambda,\mu) \to E_\lambda$ are vertical lines. When we project $\C^2-\{0\}$ to $H(\mu,\lambda)$, we identify vertical lines $x = c, x = ac, x = a^2c$ etc. Hence, the image of a vertical line in $H(\mu,\lambda)$ is a non-closed horn-like subspace with the curve $E_\mu$ in its closure.

\begin{proposition}
\label{isomorphism and cobordance}
    Let $L\to E$ and $L'\to E'$ be two line bundles on elliptic curves. Then
    \begin{enumerate}
        \item $L$ and $L'$ are Hopf dual if and only if there exists a biholomorphism $\Tot(L) - 0_L$ and $\Tot(L') - 0_{L'}$ sending a neighborhood of $0_L$ to a neighborhood of infinity in $\Tot(L')$.
        \item $L$ and $L'$ are analytically cobordant if and only if there exists a biholomorphism $\Tot(L) - 0_L$ and $\Tot(L') - 0_{L'}$ sending a neighborhood of $0_L$ to a neighborhood of $0_{L'}$.
    \end{enumerate}
\end{proposition}

\begin{proof}
    The second statement follows immediately from the first, so we will only prove the first one. We saw in (\ref{Hopf-duals}) that Hopf dual line bundles satisfy the condition of the theorem. Conversely, suppose there is a biholomorphism $\Tot(L) - 0_L$ and $\Tot(L') - 0_{L'}$ as in the proposition. Let us glue $\Tot(L)$ and $\Tot(L')$ by this isomorphism. The result is a compact Hausdorff surface compactifying $\Tot(L)$. By Enoki's theorem, it is either a ruled or a Hopf surface, hence the claim.
\end{proof}

\subsubsection{}
By Proposition \ref{isomorphism and cobordance}, analytic cobordance is an equivalence relation on the set of pairs $(E,L)$ where $E$ is an elliptic curve and $L$ is a line bundle on $E$. However, if $L$ is analytically cobordant to $L'$ and $L''$ through primary Hopf surfaces, then $L'$ and $L''$ can be analytically cobordant through a secondary Hopf surface.

\subsubsection{} The surface $M = \Tot(L)$ carries a holomorphic symplectic form $\sigma$ with a simple pole at the zero section. Indeed, $$K_M = \pi^*K_E \otimes K_{M/E} = K_{M/E} = \pi^*L^*.$$ 
The second isomorphism holds because $K_E$ is trivial. The pullback $\pi^*L$ has the tautological section, which vanishes on $0_L$; thus $K_M = \pi^*L^*$ has a nowhere zero section with a simple pole along $0_L$.

\subsubsection{} \label{cobord-symplect} Consider non-torsion analytically cobordant line bundles $L$ and $L'$. Let $\sigma$ and $\sigma'$ be the holomorphic symplectic forms on $\Tot(L) - 0_L$ and $\Tot(L') - 0_{L'}$ The pullback of $\sigma'$ under the holomophic isomorphism $f\colon \Tot(L) - 0_L \to \Tot(L') - 0_{L'}$ equals $c\sigma$ for a non-zero constant $c$. Indeed, $f^*\sigma' = g\sigma$ for a holomorphic function $g$ on $\Tot(L) - 0_L$. By \cite[Lemma 2.2]{Koike_Uehara_non_proj}, all holomorphic functions on $\Tot(L) - 0_L$ are constant.


\subsection{Hopf transforms}

\subsubsection{} A tubular neighborhood of a complex submanifold is in general not biholomorphic to a neighborhood of the zero section of its normal bundle. For example, a smooth cubic $C \subset \P^2$ can be deformed to a non-isomorphic curve, whereas all deformations of the zero section of $\Tot(\nu_{C/\P^2})$ are isomorphic to $C$. Therefore, neighborhoods of $C$ in $\P^2$ and in $\Tot(\nu_{C/\P^2})$ cannot be biholomorphic. However, analogues of the tubular neighbourhood theorem are known in certain situations.

\begin{definition}\label{diophantine-definition}
    Let $L \to E$ be a degree-zero line bundle on an elliptic curve, and $d$ an translation-invariant metric on the Picard variety $\mathrm{Pic}^0(E)$. Suppose that $$-\log d(\O_E, L^{\otimes n}) = O(\log n).$$ Then $L$ is called {\bf Diophantine}.
\end{definition}

The Diophantine property does not depend on the choice of the metric $d$. The set of Diophantine line bundles on $E$ is the complement to a measure-zero set in $\mathrm{Pic^0}(E)$.

\begin{theorem}[Arnold--Ueda theorem {\cite[4.3]{Arnold}} {\cite[Th. 3]{Ueda}}]
\label{Arnold}
    Let $S$ be a complex surface and $E \subset S$ an elliptic curve. If the normal bundle $\nu_{E/S}$ is non-torsion and of degree zero, then the formal neighbourhoods of $E$ in $S$ and in $\Tot(\nu_{E/S})$ are isomorphic. If $\nu_{E/S}$ is Diophantine, then this isomorphism extends to a biholomorphism of an analytic neighborhood of $E$ in $S$ with a neighborhood of $E$ in $\Tot(\nu_{E/S})$.
\end{theorem}

\subsubsection{} The explicit description in \cite[4.3]{Arnold} implies that if $L\to E$ and $L'\to E'$ are analytically cobordant through a primary Hopf surface and $L$ is Diophantine, then so is $L'$. It follows easily from this and \ref{secondary-with-non-torsion} that the Diophantine property is preserved by any analytic cobordism.


\subsubsection{} Arnold--Ueda theorem fails in the case of torsion line bundles. A counterexample is a fiber $F$ of a non-isotrivial elliptic fibraton. It has a holomorphically trivial normal bundle, yet, its neighbors are not isomorphic to $F$.

\begin{definition}
    \label{Hopf_transform}
    Let $E \subset S$ be an elliptic curve on a surface $S$ and $L$ its normal bundle. Assume that $\deg_EL = 0$ and $L$ satisfies the Diophantine condition. Let $L' \to E'$ be a bundle analytically cobordant to $L \to E$ (Definition \ref{def of duals and cobordant}). Pick a tubular neigbourhood of $E \subset S$ as in Theorem \ref{Arnold}, throw away $E$, and glue the tubular neighbourhood of the zero section of $L' \to F$ through a holomorphic isomorphism between $\Tot(L)-0_L$ and $\Tot(L')-0_{L'}$. The resulting surface $S'$ is called the {\bf Hopf transform} of $S$ in $E$ by $L'\to E'$. The image of the curve $E'$ in the Hopf transform is called the {\bf graft}.
\end{definition}


\subsubsection{} The new surface $S'$ may not be algebraic or even K\"ahler even if $S$ is. For example, take $S = \P(\O_E\oplus L)$ and embed $E$ as one of the sections. Then $S'$ is a Hopf surface.


\subsubsection{} The Hopf transform of a given surface in a given curve can be made in a countable number of non-isomorphic ways due to Theorem \ref{embeddings into all hopfs}. Nevertheless, the result of a Hopf transform depends only on the choice of an analytically cobordant line bundle and not on an isomorphism between a neighbourhood of $E$ in $S$ and in $\Tot(L)$. The proof relies on the following Lemma.


\begin{lemma}
\label{open embeddings}
    Let $L$ be a non-torsion degree-zero line bundle on an elliptic curve $E$. For $r>0$, denote by $W_r$ the subset of vectors in $\Tot(L)$ of length less than $r$. Consider an open embedding $j\colon W_r \to W_R$ identical on $E$, where $r,R\in \R_{>0}\sqcup\infty$. Then $j$ is the fiberwise multiplication by a constant $c\in \C^*$ such that $|c|\le R/r$.
\end{lemma}
\begin{proof}
{\bf Step 1.} The space $W_r$ (resp.\! $W_R$) is isomorphic to $\C\times B_r/\Lambda$ (resp. $\C\times B_R/\Lambda$) where $\Lambda:= \Z + \Z\cdot\tau$ acts as follows: 
$$
\gamma\cdot(x,y) = (x+\gamma, \rho(\gamma)y).
$$
Here $B_r$ (resp. $B_R$) is the open disk of radius $r$ (resp. $R$) and $\rho\colon \Lambda\to U(1)$ is the monodromy representation. By the universal property of universal covers, we can lift $j$ to an open embedding $J\colon \C\times B_r\to \C\times B_{R}$. The map $J$ descends to quotients, hence for every $\gamma\in \Lambda$ there exists $\gamma'\in \Lambda$ such that $J(\gamma\cdot(x,y)) = \gamma'\cdot J(x,y)$. When $y=0$ the map $J$ is the identity, hence $\gamma' = \gamma$. We conclude that $J$ is $\Lambda$-equivariant.

\hfill


{\bf Step 2.} Write $J$ as $(f,g)$, where $f\colon \C\times B_r \to \C$ and $g\colon \C\times B_{r}\to B_{R}$. We will see in this step that $g(x,y) = cy$ for a constant $c\in\C^*$. For a fixed $y\in B_r$, the map $g$ is a holomorphic function from $\C$ to $B_{R}$, hence constant. We conclude that $g = g(y)$. The $\Lambda$-equivariance of $J$ implies that $g(\rho(\gamma)y) = \rho(\gamma)g(y)$. The image of $\Lambda$ in $U(1)$ is dense because $L$ is non-torsion. Therefore, $g$ is $U(1)$-equivariant, and our claim follows.

\hfill


{\bf Step 3.} By $\Lambda$-equivariance of $J$ we have that 
\begin{equation}
\label{aaa}
f(x+\gamma, \rho(\gamma)y) = f(x,y) + \gamma.
\end{equation} 
Let as differentiate the equation (\ref{aaa}) $k$ times with respect to $y$. We obtain
$$
\rho^k(\gamma)\frac{\di^k f}{\di y^k}(x+\gamma,\rho(\gamma)y) = \frac{\di^k f}{\di y^k}(x,y)
$$
Set $y=0$. The function $\left|\frac{\di^k f}{\di y^k}(x,0)\right|$ is invariant under shifts by $\gamma\in \Lambda$, hence bounded. We conclude that $\frac{\di^k f}{\di y^k}(x,0)$ is constant. The representation $\rho^k$ is non-trivial for every $k$, hence $\frac{\di^k f}{\di y^k}(x,0)$ vanishes. We conclude that $f$ does not depend on $y$. Hence $f(x,y) = f(x,0) = x$. We showed that $J(x,y) = (x,cy)$, hence the claim.
\end{proof}


\begin{corollary}
Let $E\subset S$ be as in Definition \ref{Hopf_transform}. Fix a line bundle $L'$ analytically cobordant to $L\to E$ (Definition \ref{def of duals and cobordant}). Then the Hopf transform of $S$ in $E$ by $L'$ is well-defined. Namely, consider two biholomorphisms $\phi_1\colon U_1\stackrel{\sim}\to W_1$ and $\phi_2\colon U_2\stackrel{\sim}\to W_2$  between neighborhoods $U_1$ and $U_2$ of $E$ in $S$ and neighborhoods $W_1$ and $W_2$ of $E$ in $\Tot(L)$. Assume that $\phi_1|_E = \phi_2|_E$. Then the Hopf transforms $S_1$ and $S_2$ of $S$ obtained through $\phi_1$ and $\phi_2$ respectively are biholomorphic.   
\end{corollary}


\begin{proof}The Hopf transforms $S_i$, $i=1,2$, will not change if we shrink $U_i$. Hence we may assume that $U_1\subset U_2$ and $\phi_1(U_1) = W_r$, $\phi_2(U_2)=W_{R}$ for some $r,R\in\R_{>0}$. The map $j:= \phi_2|_{U_2}\circ\phi_1^{-1}$ is an open embedding $W_r\to W_R$. Lemma \ref{open embeddings} implies that $j$ is the multiplication by a constant $c\in\C^\times $. Embed $\Tot(L^{-1})$ into a Hopf surface $H(\lambda,\mu)$ such that $H(\lambda,\mu)-E_\lambda\simeq \Tot(L^{-1})$ and $H(\lambda,\mu)-E_\mu\simeq \Tot(L')$. The biholomorphism $j$ extends to the automorphism $\diag(1,c)$ of $H(\lambda,\mu)$ by \ref{multiplication by a constant}. This automorphism of $H(\lambda,\mu)$ induces a biholomorphism $S_1\to S_2$.
\end{proof}


\subsubsection{} We only define the Hopf transform 
in elliptic curves with Diophantine normal bundle. The same definition 
works for any square-zero elliptic curve with
a holomorphic tubular neighbourhood. 

\hfill

An equivalent formulation of Enoki's theorem (Theorem \ref{enoki-classification}) 
is that every compactification of $\Tot(L)$ is a Hopf 
transform of $\mathbb P(\O\oplus L)$ in the infinity section. 
This statement generalizes partially 
to other complex surfaces.


\begin{theorem}
\label{all compactifications}
Suppose $M$ is a complex surface realizable as the complement to a
square-zero elliptic curve $E$ with a Diophantine normal bundle $L$ in a 
compact surface $S$. Then every minimal analytic compactification of $M$ is 
a Hopf transform of $S$ in $E$.
\end{theorem}

\begin{proof}
The Diophantine condition implies that the neighborhood of infinity in $X$ 
is biholomorphic to a neighborhood of the zero section in $\Tot(L)$ with 
the zero section removed (Theorem \ref{Arnold}). Hence, every partial compactification of $\Tot(L) - 0_L$ 
near the zero section produces a compactification of $M$ and vice versa. 
By Theorem \ref{embeddings into all hopfs} and Enoki's theorem, every partial compactification of $\Tot(L)-0_L$ is a Hopf transform at $0_L$.
\end{proof}


\begin{corollary}
    A Hopf transform of a Hopf transform in its graft is either a Hopf transform, or the initial surface. Inverse of a Hopf transform is a Hopf transform.
\end{corollary}
\begin{proof}
The composition of Hopf transforms of $S$ in $E$ is a compactification of $S - E$, hence a Hopf transform (Theorem \ref{all compactifications}).
\end{proof}

\begin{corollary}
\label{algebraic_structures}
    Let $S$ be a projective surface with a square-zero elliptic curve $E$ with a Diophantine normal bundle. Suppose that each Hopf transform of $S$ in $E$ is projective. Then the set of algebraic structures on $M:=S-E$ is countable.
\end{corollary}
\begin{proof}
    By Theorem \ref{all compactifications}, every compactification of $M$ is a Hopf transform of $S$ in $E$. Thus every compactification induces an algebraic structure on $M$. Conversely, every algebraic structure is induced from a compactification. Therefore, the set of compactifications $(S',E')$ of $M$ is countable (Theorem \ref{embeddings into all hopfs}). Moreover, the set of possible $E'$ is also countable. So, it is enough to prove that two compactifications $(S',E')$ and $(S'',E'')$ such that $E'\not\simeq E''$ induce distinct algebraic structures on $M$. An algebraic isomorphism between $\phi\colon S'-E'\to S''-E''$ would induce a birational morphism $\phi\colon S' \DashedArrow[->,densely dashed    ]  S''$. A birational map of smooth surfaces cannot contract a non-rational curve, hence $E'\simeq \phi(E')\simeq E''$, contradiction.
\end{proof}


\section{Surfaces with square-zero elliptic curves}\label{square-zero-section}


\subsection{Kodaira dimension}


\subsubsection{} 
\label{kodaira_not_two}
Let $S$ be a complex surface (not necessarily projective) containing a square-zero elliptic curve $E$. Then its Kodaira dimension $\kappa(S)$ is at most one. Indeed, suppose $\kappa(S) = 2$. Then the map $\phi\colon S \to \P^N$ induced by the linear system $|nK_S|$ is birational onto its image for sufficiently large $n$. Every square-zero elliptic curve $E$ satisfies $K_S\cdot E = 0$. Hence the map $\phi$ must contract $E$. Yet, a birational map of surfaces cannot contract a square-zero curve.


\begin{proposition}
    Let $S$ be a surface with square-zero elliptic curve $E\subset S$ whose normal bundle $\nu_{E/S}$ is non-torsion. Then $\kappa(S)=-\infty$.
\end{proposition}
\begin{proof}
    Observe that $\nu_{E/S}\simeq \O_E(-K_S|_E)$. We consider two cases. 
    
    {\bf Case 1: $\kappa=0$.} Every minimal surface with $\kappa = 0$ has torsion canonical bundle, hence the statement is trivial in this case. If $S$ is non-minimal, then $K_S = \sum E_i +D$, where $E_i$'s are exceptional curves and $D$ is a torsion divisor. We get that $K_S|_E$ is either torsion or has positive degree for every curve $E\subset S$. 
    
    {\bf Case 2: $\kappa=1$.} The linear system $|nK_S|$ for $n\gg0$ induces a morphism $\phi \colon S\to \P^N$ whose image is a curve. It is an elliptic fibration. Elliptic curves in the fibers of $\phi$ have torsion normal bundle, so we can consider only horizontal curves. For every horizontal curve $C$, we have $n K_S|_C = (\phi|_C)^*H$, where $H$ is the hyperplane section. Therefore, $K_S\cdot C$ is positive. 
\end{proof}


\subsubsection{} \label{trichotomy} A surface of negative Kodaira dimension is one of the following:

\begin{itemize}
    \item rational;
    \item birational to a ruled surface;
    \item of class VII, i.e., a non-K\"ahler surface with $\kappa=-\infty$ and $b_1=1$.
\end{itemize}

If a blow-up of a ruled surface $S$ contains an elliptic curve, then $S$ is ruled over an elliptic curve.

\begin{proposition}\label{anti-canon-means-P2}
    Let $E \subset S$ be an anti-canonical elliptic curve with a degree-zero non-torsion normal bundle $\nu_{E/S}$. Assume  $b_1(S)=0$. Then $S$ is a blow-up of a length nine subscheme in $\P^2$.
\end{proposition}
\begin{proof}
    The canonical class of $S$ is anti-effective, hence all plurigenera $p_n := h^0(K_S^n)$ of $S$ vanish. Since $b_1(S) = 0$, the irregularity $q = h^1(\O_Y)$ vanishes as well. Castelnuovo theorem implies that $S$ is rational, in particular, projective \cite[VI (3.4)]{Barth_Hulek_Peters_Van_de_Ven}. Its minimal model is either $\P^2$ or a Hirzebruch surface $\F_n = \P\left(\O_{\P^1}\oplus\O_{\P^1}(-n)\right)$. The image of $E$ in a minimal model is an irreducible anti-canonical curve: indeed, $E$ is anti-canonical, so it intersects each $(-1)$-curve transversely at one point. Since $K_S^2=0$, $S$ is either a blow-up of nine points in $\P^2$ or of eight points in $\F_n$. 

\hfill

    The Hirzebruch surface $\F_n$ contains a section $C_n$ of square $-n$. Projection formula yields $-K_{\F_n}\cdot C = 2-n$, thus for $n>2$ an anti-canonical curve contains $C$ and is reducible. Therefore, any minimal model of $S$ is $\P^2$, $\F_0 = \P^1\times\P^1$, or $\F_2$. An anti-canonical curve in $\F_2$ does not intersect $C_2$. A blow-up of a point away from $C_2$ is isomorphic to a blow-up of a length-two subscheme in $\P^2$. A blowup of any point on $\F_0 = \P^1\times\P^1$ is isomorphic to a blowup of two points on $\P^2$. Thus $S$ is a blow-up of $\P^2$.
\end{proof}


\subsection{Rational surfaces}


Let us start with the following lemma.

\begin{lemma}\label{excision-lemma}Let $S$ be a surface, and $C \subset S$ a smooth curve. Then the natural map $H_1(S-C,\Q) \to H_1(S,\Q)$ is surjective. If $S$ is K\"ahler, then it is also injective.
\end{lemma}
\begin{proof}
    Look at the long exact sequence of cohomology associated with the decomposition $S = C \sqcup (S-C)$:
        $$\dots \to H^2(S,\Q) \to H^2(C,\Q) \to H^3_c(S-C,\Q) \to H^3(S,\Q) \to H^3(C,\Q) = 0.$$
    If $S$ is K\"ahler, the class of $C$ in $H_2(S)$ is non-trivial. Therefore the map $H^2(S,\Q)\to H^2(C,\Q)$ is surjective, and $H_1(S-C,\Q) \simeq H^3_c(S-C,\Q)\simeq H^3(S,\Q)\simeq H_1(S,\Q)$.
\end{proof}

The next statement follows easily from Lemma \ref{excision-lemma}, and we omit its proof.

\begin{corollary}
    \begin{enumerate}
        \item Let $S$ be a rational surface, $C \subset S$ a smooth curve. Then $h_1(S-C)=0$.
        \item Let $S$ be a blow-up of a ruled surface over an elliptic curve, $C \subset S$ a smooth curve. Then $h_1(S-C)=2$. 
    \end{enumerate}
\end{corollary}

\begin{corollary}\label{Hopf-transform-preserves-rationality}
    \begin{enumerate}
        \item A Hopf transform of a rational surface is rational.
        \item A Hopf transform of a surface birational to a ruled surface is birational either to a ruled surface or a Hopf surface.
    \end{enumerate}
\end{corollary}
\begin{proof}
    Let $S$ be a rational surface. Lemma \ref{excision-lemma} implies that $h_1(S-C) =0$. Consider a Hopf transform $(S',C')$ of $S$ in $C$. The space $S'-C'$ is biholomorphic to $S-C$, hence $h_1(S'-C') = h_1(S-C)$. Lemma \ref{excision-lemma} implies that $h_1(S'-C') \geqslant h_1(S')$. Thus a Hopf transform of a rational surface satisfies $h_1(S') = 0$. The classification in \ref{trichotomy} implies that $S'$ is rational. 
    
    A Hopf transform of a blow-up of a ruled surface cannot be rational by the previous statement. By \ref{trichotomy} its Hopf transform is either a blow-up of a ruled surface or a Hopf surface.
\end{proof}

\subsubsection{} Choose an elliptic curve $E$ in $\P^2$ and nine points $p_1,...p_9\in E$ (points are allowed to collide). Denote the hyperplane section of $\P^2$ as $H$. Assume that the line bundle $\O_E(3H-p_1 -...-p_9)$ satisfies the Diophantine condition (Definition \ref{diophantine-definition}). This line bundle is isomorphic to the normal bundle to the strict transform of $E$ in the blow-up $X$ in $p_1,...p_9$. Therefore, a Hopf transform $Y$ of $X$ in $E$ is well-defined.

\begin{theorem}
\label{hopf_of_blow_up9}
    Any Hopf transform $Y$ of $X$ is the blow-up of $\P^2$ in a length nine subscheme $\gamma$. The graft is the strict preimage in $Y$ of the elliptic curve $F \subset \P^2$ passing through $\gamma$.
\end{theorem}
\begin{proof} 
{\bf Step 1.} The surface $(X,E)$ is log Calabi--Yau, that is, $K_X+E=0$. This implies existence of a meromorphic symplectic form on $X$ with a simple pole along $E \subset X$. We know from \ref{cobord-symplect} that a Hopf transform is a gluing along a holomorphic symplectomorphism, hence $Y$ carries a symplectic form with a simple pole along $F \subset Y$. Thus $K_Y+F=0$.

\hfill

{\bf Step 2.} By Corollary \ref{Hopf-transform-preserves-rationality}, $Y$ is rational. Moreover, its anti-canonical divisor is effective. By Proposition \ref{anti-canon-means-P2}, $Y$ is a blowup of a length-nine subscheme in $\P^2$, and $F \subset Y$ is the strict transform of the only plane cubic passing through it. 
\end{proof}

The following question remains open:

\begin{problem}
   Determine the nine points $q_1,...q_9\in \P^2$ such that $Y \simeq \Bl_{q_1,..q_9}(\P^2)$. 
\end{problem}


\begin{theorem}\label{many-algebraic-structures}
    Let $X$ be the blowup of $\P^2$ in a very generic length-nine subscheme $\gamma$, and $E \subset X$ the strict transform of the plane cubic passing through $\gamma$. Then the analytification of $X-E$ admits countably many algebraic structures.
\end{theorem}
\begin{proof}
The claim follows from Corollary \ref{algebraic_structures} because every Hopf transform of $X$ is projective.

\end{proof}


\subsubsection{} There are rational surfaces containing a not anti-canonical elliptic curve of square zero. Let $Q \subset \P^2$ be a plane quartic with two nodes. Blow them up; the strict transform $\tilde{Q}$ has square $8$. By blowing up eight more points on $\tilde{Q}$, one gets a square-zero elliptic curve on $\P^2$ blown up in ten points. Notice that `having a node at a given point' is a codimension three condition, thus plane quartics with two nodes at $p,q\in\P^2$ form an eight-dimensional space. Thus, a generic tuple of ten points with two distinguished ones determines a unique plane quartic with two nodes through it.

\subsubsection{} By \cite[Lemma 2.2]{Koike_Uehara_non_proj}, punctured neighbourhood of a square-zero elliptic curve with non-torsion normal bundle carries no nonconstant functions. Hence none of the surfaces $X-E$ from Theorem \ref{many-algebraic-structures} is Stein.


\subsection{Surfaces of class VII}


All surfaces of class VII that contain an elliptic curve with a non-torsion degree-zero normal bundle are Hopf surfaces, as we prove below.

    \begin{proposition}\label{classification-for-class-vii}
        Let $S$ be a surface of class VII and $E \subset S$ a smooth elliptic curve with non-torsion normal bundle $\nu_{E/S}$ of degree zero. Then $S$ is a primary Hopf surface $H(\lambda,\mu)$ with $\lambda^n\neq\mu^m$ for any $n,m\in\Z$, or a secondary Hopf surface $H(\lambda,\mu)/\mu_n$ from \ref{secondary-with-non-torsion}, or a blow-up of such a surface away from $E$.
    \end{proposition}
    \begin{proof}
        Every curve on a smooth non-K\"ahler surface has a non-positive square \cite[Ch.\:IV, Thm.\:2.14]{Barth_Hulek_Peters_Van_de_Ven}. Applying this fact to a minimal model of $S$, we obtain that a curve of square zero does not intersect exceptional curves.

        If $E \subset S$ is a nonsingular square-zero curve on a minimal surface $S$ of class VII, then $S$ is a Hopf surface \cite[Proposition 4.12]{Enoki}. Hopf surfaces that contain an elliptic curve with non-torsion normal bundle were classified in \ref{generalities on primary hopfs} and \ref{secondary-with-non-torsion}.
    \end{proof}


\section{Analytic Grothendieck group}\label{grothendieck-section}

\stepcounter{subsection}

\begin{definition}
    The {\bf Grothendieck group $K_0(\V_\C)$ of complex algebraic varieties} is the abelian group generated by classes of $\C$-varieties modulo scissor relations:
    $$[X - Y] = [X]-[Y],$$
    where $Y \subset X$ is a closed algebraic subvariety.\footnote{It is actually a ring with multiplication $[U]\cdot[V] = \left[{U}\times_k{V}\right]$, but we do not use it.}
\end{definition}

We introduce an interesting quotient of this group.

\begin{definition}\label{analytic-grothendieck-definition}
    The {\bf analytic Grothendieck group of varieties} $K_0^{an}$ is the quotient of $K_0(\V_\C)$ by differences $[U]-[V]$ where $U$ and $V$ are biholomorphic varieties.
\end{definition}

\begin{proposition}
\label{analytic K zero}
    The classes of any two elliptic curves in $K_0^{an}$ are equal.
\end{proposition}
\begin{proof}
    {\bf Step 1.} Suppose that $E, F$ are two elliptic curves and $L\to E$, $L' \to F$ are two analytically cobordant Diophantine line bundles. Then $[E] = [F] \in K_0^{an}$. Indeed, let $X$ be a blow-up of nine points in $\P^2$ such that the cubic passing through them is isomorphic to $E$, and the normal bundle of its strict transform is isomorphic to $L$. Even in $K_0(\V_{\C})$, one has $[X] = \mathbb{L}^2 + 10\mathbb{L} + [\mathrm{pt}]$, where $\mathbb{L}$ is the class of the affine line. By Theorem \ref{hopf_of_blow_up9}, the Hopf transform $Y$ of $X$ in $E$ by $L'\to F$ is also a blow-up of $\P^2$ in a length nine subscheme. Thus $[Y] = \mathbb{L}^2 + 10\mathbb{L} + 1 = [X]$. It follows from the relations $[X] = [E] + [X-E]$ and $[Y] = [F] + [Y-F]$ that
    $$
    [E]^{an} = [X]^{an} - [X-E]^{an} = [Y]^{an} - [Y-F]^{an} = [F]^{an} \in K_0^{an}.
    $$

    {\bf Step 2.} For any pair of elliptic curves $E$, $F$, there exists a primary Hopf surface $H = H_{E,F}$ containing $E$ and $F$. For a fixed $E$, the set of curves $F$ for which the normal bundle $\nu_{E/H_{E,F}}$ is not Diophantine is meagre in the moduli space of elliptic curves. Thus for arbitrary two curves $E$, $E'$, the complements of the corresponding meagre sets have nonempty intersection. Let $F$ be any curve from this intersection. Step 1 implies that $[E]^{an} = [F]^{an}$ and $[F]^{an} = [E']^{an}$, hence the claim.
\end{proof}

\subsubsection{} \label{Bittner-theorem} Although the classes of elliptic curves in $K_0^{an}$ 
coincide, the classes of different elliptic curves in $K_0(\V_\C)$ are distinct, as 
follows from Franziska Bittner's description of $K_0(\V_\C)$ in \cite{Bittner}. She proved 
that $K_0(\V_\C)$ is generated by classes of smooth projective varieties modulo relations 
\begin{equation}\label{blow-up}
[X] - [Y] = [Bl_YX] - [E],
\end{equation}
where $Y$ is a smooth subvariety of a smooth projective variety $X$ 
and $E\subset Bl_YX$ 
is the exceptional divisor of the blow-up of $X$ in $Y$. This result is very powerful as it 
enables us to construct motivic measures (i.e., maps from $K_0(\V_\C)$ to an abelian group) 
by checking only the relation (\ref{blow-up}). For example, 
consider the functional on smooth projective varieties sending a variety to its stable 
birationality class. This functional descends to a motivic measure
$$
K_0(\V_\C) \to \Z[\mathrm{SB}]
$$
with values in the free abelian group generated by stable birationality classes. Indeed, 
$Bl_YX$ is birational to $X$ and $E$ is stably birational to $Y$, so the relation (\ref{blow-up}) 
is preserved. We see that non-stably birational varieties have different classes in $K_0(\V_\C)$. 
Two curves are stably birational if and only if they are isomorphic \cite[V Ex.\! 2.1]{Hartshorne}. In particular, 
non-isomorphic elliptic curves have different classes in $K_0(\V_\C)$. 

\subsubsection{} Topological Euler characteristic is a homomorphism $K_0^{an} \to \Z$. However, we were not able to find more interesting analytic motivic measures. This motivates the following

\begin{problem}
    Is the class of an elliptic curve in $K_0^{an}$ trivial? If so, is $K_0^{an}$ isomorphic to $\Z$?
\end{problem}

\paragraph{Acknowledgements.} We are grateful to Piotr Achinger, to whom we owe the idea of Section \ref{grothendieck-section}. Many thanks to Robert Friedman, Micha\l~Kapustka, Takayuki Koike, Aleksandr Petrov, and Giulia Sacc\`a for insightful discussions, and to Nathan Chen, Andr\'es Fern\'andez Herrero, and Morena Porzio for reading an early draft of the paper.

\bibliographystyle{alpha} 
\bibliography{mybib.bib} 

\begin{multicols}{2}
\noindent {\sc {Anna Abasheva} \\
Columbia University,\\
Department of Mathematics, \\
2990 Broadway, \\
New York, NY, USA}\\
{\tt anabasheva(at)math.columbia.edu}

\columnbreak

\noindent {\sc {Rodion D\'eev}\\
Institute of Mathematics, \\ Polish Academy of Sciences,\\
\'Sniadeckich 8, \\ 
Warsaw, Poland}\\
{\tt rdeev(at)impan.pl}

\end{multicols}

\end{document}